\documentclass[12pt]{article}
\usepackage{a4}
\usepackage{hyperref}
\usepackage{mathrsfs}
\usepackage{mathtools}
\usepackage{amsmath}
\usepackage{amssymb}
\usepackage{amsfonts}
\usepackage{graphicx}
\usepackage{color}
\usepackage{pstricks,color,rotating}

\newtheorem{thm}{Theorem}

\newtheorem{cor}[thm]{Corollary}
\newtheorem{df}[thm]{Definition}

\newtheorem{ex}[thm]{Example}
\newcommand{\bdf}{\begin{df} \begin{rm}}
\newcommand{\edf}{\end{rm} \end{df}}

\newenvironment{proof}{{\bf Proof.}}{\hspace*{\fill} \rule{2mm}{2mm} \par \hspace{0.1mm}}

\title{On the Hyper Zagreb index of certain \\ generalized thorn graphs}
\author{C.Natarajan, S.Balachandran and S.K.Ayyaswamy\\
Department of Mathematics,\\ School of Humanities and Sciences,\\
SASTRA Deemed to be University,\\ Thanjavur-613 401,\\
Tamilnadu, India.\\
\small{natarajan\_c@maths.sastra.edu,}\\ \small{bala\_maths@rediffmail.com, sjcayya@yahoo.co.in} }
\date{}
\begin{document}

\maketitle
\begin{abstract}
Let $G=(V,E)$ be a graph with $n$ vertices and $m$ edges. The hyper Zagreb index of $G$, denoted by $HM(G)$, is defined as $HM(G) =\sum\limits_{uv \in E(G)}\left[d_{G}(u)+d_G(v)\right]^{2}$ where $d_G(v)$ denotes the degree of a vertex $v$ in $G$.  In this paper we compute the hyper Zagreb index of certain generalized thorn graphs.
\end{abstract}
\section{Introduction}

	Mathematical calculations are absolutely necessary to explore important concepts in
chemistry. In mathematical chemistry, molecules are often modeled by graphs named
“molecular graphs”. A molecular graph is a simple graph in which vertices are the
atoms and edges are bonds between them. By IUPAC terminology, a topological
index is a numerical value for correlation of chemical structure with various physical
properties, chemical reactivity or biological activity.
	
	In chemical graph theory, a graphical invariant is a number related to a graph which is
	structurally invariant. These invariant numbers are also known as the topological indices. The well-known Zagreb indices are one of the oldest
	graph invariants firstly introduced by Gutman and Trinajesti\'{c}
	\cite{gt} more than forty years ago, where Gutman and Trinajsti\'{c} examined the dependence of
	total $\pi$-electron energy on molecular structures, and this was
	elaborated on in \cite{grtw}.
	
	Throughout this paper we consider only simple and connected graphs. For a graph
	$G=(V,E)$ with vertex set $V=V(G)$ and edge set $E=E(G)$, the degree
	of a vertex $v$ in $G$ is the number of edges incident to $v$ and
	denoted by $d_G(v)$. 

    For a (molecular) graph $G$, the first
	Zagreb index $M_1(G)$ and the second Zagreb index $M_2(G)$ are,
	respectively, defined as follows:
	$$M_1=M_1(G)=\sum\limits_{v\in
		V(G)}d^2_G(v), \,\,\,\, M_2=M_2(G)=\sum\limits_{uv\in
		E(G)}d_G(u)d_G(v).$$
	Also, $M_1(G)=\sum\limits_{uv\in
		E(G)}[d_G(u)+d_G(v)]$.
For more details on these indices see the
	recent papers \cite{ai, aig, ds, dxg, dsab, gd, hr, ll, lld, x, xdb, wx} and the references therein.
	
	Mili\v{c}evi\v{c} et al. \cite{mkt} in 2004 defined reformulated the Zagreb indices in terms of edge degrees
instead of vertex degrees as:
\begin{align}
EM_1 (G) = \sum_{e \in E(G)} d(e)^{2}  \nonumber
\end{align}
and 
\begin{align}
EM_2 (G) = \sum_{e\thicksim f } d(e) d(f)   \nonumber
\end{align}
where $d(e)$ represents the degree of the edge $e$ in $G$, which is deﬁned by $d(e) = d(u) + d(v) - 2$ with $e = uv$, and $e \thicksim
 f$ represents the fact that the edges $e$ and $f$ are adjacent.
			
	Shirdel et al. in \cite{srs} defined hyper Zagreb index, as: 
	\begin{align}
	HM(G)= \sum_{uv \in E(G)}\left(d_G(u)+d_G(v)\right)^2    \nonumber
	\end{align}
	and discussed some graph operations of the hyper
Zagreb index in \cite{srs}. 
	
Gao et al. \cite{gjf}presented exact expressions for the
hyper-Zagreb index of graph operations containing cartesian product and join of $n$
graphs, splice, link and chain of graphs. Veylaki et al. \cite{vnt} calculated third and hyper
Zagreb coindices of graph operations containing the Cartesian product and composition.

Vuki\v{c}evi\v{c} et al. \cite{vg} calculated the modified Wiener index of thorn graphs. In \cite{zgv} Zhou et al. found an explicit formula to calculate the variable Wiener index of thorn graphs. Zhou et al.\cite{zv} derived the expression for Wiener-type polynomials of thorn graphs. Heydari et al. \cite{hg} calculated terminal Weiner index of thorn graphs.   Nilanjan De et al. computed F-index of t-thorn graphs in \cite{nma}.  K.M.Kathiresan et al.\cite{kp} obtained some bounds on the Wiener index of certain generalized thorn graphs.  Venkatakrishnan et al. \cite{vbk} computed eccentric connectivity index of the same structures.   \\

\section{Hyper Zagreb index of certain generalized thorn graphs}

 The $t$-thorn graph of a graph $G$, denoted by $G^{t}$, is a graph obtained by joining $t$ copies of pendent edges known as thorns to each vertex of $G$.
\par In this section, we compute the Hyper Zagreb index of the following four types of generalized thorn graphs.
 \\ Let $G$ be a graph of order $n$ and size $m$ respectively.  Let $V(G)=\{v_1,v_2,\cdots,v_n\}$.
\\ \noindent{\bf Generalized thorn graph of Type-I}
\\  Attach $t_i$ copies of a path of order $r \geq 2$ at each vertex $v_i$ of $G$ by identifying the vertex $v_i$ as the initial vertex of such paths.  The resulting graph thus obtained is denoted by $G_P$.
\\ \noindent{\bf Generalized thorn graph of Type-II}
\\ Attach $t_i$ copies of a cycle of length $r$ to each vertex $v_i$ of $G$ by identifying $v_i$ as a vertex in each cycle.  The resulting graph thus obtained is denoted by $G_C$.
\\ \noindent{\bf Generalized thorn graph of Type-III}
\\  Attach $t_i$ copies of a complete graph $K_r$ of order $r \geq 3$ to each vertex $v_i$ of $G$ by identifying $v_i$ as a vertex in $K_r$.  The graph thus obtained is denoted by $G_K$.
\\ \noindent{\bf Generalized thorn graph of Type-IV}
\\ Attach $t_i$ copies of a complete bipartite graph $K_{r,s}$ to each vertex $v_i$ of $G$ by identifying $v_i$ as a vertex in a partition of $K_{r,s}$ containing $r$ vertices.  The resulting graph thus obtained is denoted by $G_A$.
\\ \noindent{\bf Generalized thorn graph of Type-V}
\\ To every vertex $v_i$ of $G$, join $t_i$ copies of $C_r$ each by an edge.  The resulting graph is denoted by $G_{C'}$.
\\ \noindent{\bf Generalized thorn graph of Type-VI}
\\ To every vertex $v_i$ of $G$, join $t_i$ copies of $K_p$ each by an edge.  The graph thus obtained is denoted by $G_{K}^{'}$.
\\ \noindent{\bf Generalized thorn graph of Type-VII}
\\ To every vertex $v_i$ of $G$, join $t_i$ copies of a complete bipartite graph $K_{r,s}$ each by an edge.

\begin{thm} $HM(G_P)=HM(G)+2 \sum\limits_{v_iv_j \in E(G)}\left[d(v_i)+d(v_j)\right](t_i+t_j)+\sum\limits_{v_iv_j \in E(G)}(t_i + t_j)^{2} + (16r-35)\sum\limits_{i=1}^{n}t_i + \sum\limits_{i=1}^{n} \left[t_id(v_i)^2+t_i^{3}+4t_i^{2}+2t_i^{2}d(v_i)+4t_id(v_i)\right]$.
\end{thm}
\begin{proof} In the generalized thorn graph $G_P$ of $G$, $d_{G_P}(v_i)=d(v_i)+t_i$, for any vertex $v_i \in V(G)$.  Therefore, 
\begin{alignat*}{7} HM(G_P) ={} & \sum\limits_{v_iv_j \in E(G_P)}\left[d(v_i)+d(v_j)\right]^{2} \\
                            ={}& \sum\limits_{v_iv_j \in E(G_P)} \left[d(v_i)+t_i+d(v_j)+t_j\right]^{2} \\
+{}& \sum\limits_{i=1}^{n}t_i\left[d(v_i)+t_i+2\right]^{2}+16(r-3)\sum\limits_{i=1}^{n}t_i + 9 \sum\limits_{i=1}^{n}t_i \\
={}& HM(G)+2 \sum\limits_{v_iv_j \in E(G)}\left[d(v_i)+d(v_j)\right](t_i+t_j)\\
+{}& \sum\limits_{v_iv_j \in E(G)}(t_i+t_j)^{2}+\sum \limits_{i=1}^{n}t_i \left[d(v_i)^{2}+(t_i+2)^{2}+2d(v_i)(t_i+2)\right]+(16r-35)\sum\limits_{i=1}^{n}t_i \\
={}& HM(G)+2 \sum\limits_{v_iv_j\in E(G)}\left[d(v_i)+d(v_j)\right](t_i+t_j) \\
+{}& \sum\limits_{v_iv_j \in E(G)}(t_i+t_j)^{2}+\sum \limits_{i=1}^{n}\left[t_id(v_i)^{2}+t_i^{3}+4t_i^{2}+2t_i^{2}d(v_i)+4t_id(v_i)\right]+(16r-35)\sum\limits_{i=1}^{n}t_i.
\end{alignat*}

\end{proof}

\begin{cor} If $t_i=t$, for all $1 \le i \le n$, then $HM(G_P)=HM(G)+5tM_1(G)+8mt^2+nt^3+4nt^2+16mt+16rnt-35nt$.
\end{cor}
\begin{ex} \begin{description}
		\item[(i)] If $G\cong P_n$ and $t_i=t$ for all $1\le i\le n$, then $HM(G_P)=16n+(8m+4n)t^2+16t(m+rn)-15nt-30t-30$
\item[(ii)]	If $G \cong C_n$ and $t_i=t$ for all $1 \le i\le n$, then $HM(G_P)=16n-15nt+8mt^2+nt^3+4nt^2+16mt+16rnt$.
\end{description}
	\end{ex}
\begin{thm} $HM(G_C)=HM(G)+4 \sum \limits_{v_iv_j \in E(G)}\left[d(v_i)+d(v_j)\right](t_i+t_j)+4 \sum\limits_{v_iv_j \in E(G)}(t_i+t_j)^2 +2\sum\limits_{i=1}^{n}t_id(v_i)^{2}
+8\sum\limits_{i=1}^{n}\left[d(v_i)t_i^{2}+d(v_i)t_i+t_i^{3}+2t_i^{2}\right]+(16r-24)\sum\limits_{i=1}^{n}t_i$
\end{thm}
\begin{proof} \begin{alignat*}{9} HM(G_C)={}& \sum\limits_{v_iv_j\in E(G_C)}\left[d(v_i)+d(v_j)\right]^{2} \\
={}& \sum\limits_{v_iv_j\in E(G)}\left[d(v_i)2t_i+d(v_j)+2t_j\right]^{2} \\
={}& HM(G)+4\sum\limits_{v_iv_j\in E(G)}\left[d(v_i)+d(v_j)\right](t_i+t_j)\\
+{}& 4\sum\limits_{i=1}^{n}(t_i+t_j)^{2}+2\sum\limits_{i=1}^{n}t_i\left[d(v_i)^{2}+4d(v_i)(t_i+1)+4t_i^{2}+8t_i+4\right] \\
+{}& (16r-32)\sum\limits_{i=1}^{n}t_i\\
={}& HM(G)+4 \sum \limits_{v_iv_j \in E(G)}\left[d(v_i)+d(v_j)\right](t_i+t_j)\\
+{}& 4 \sum\limits_{v_iv_j \in E(G)}(t_i+t_j)^2 +2\sum\limits_{i=1}^{n}t_id(v_i)^{2} \\
+{}& 8\sum\limits_{i=1}^{n}\left[d(v_i)t_i^{2}+d(v_i)t_i+t_i^{3}+2t_i^{2}\right]
+ (16r-24)\sum\limits_{i=1}^{n}t_i
\end{alignat*}
\end{proof}
\begin{cor} If $t_i=t$, for all $1 \le i \le n$, then $HM(G_C)=HM(G)+10tM_1(G)+32mt^2+16mt+8nt^3+16nt^2+(16r-24)nt$.
\end{cor}
\begin{ex} \begin{description}
		\item[(i)] If $G \cong P_n$ and $t_i=t$ for all $1 \le i \le n$, then $HM(G_C)=16n-30+16nt-60t+32mt^2+16mt+8nt^3+16nt^2+16rnt$\item[(ii)] If $G \cong C_n$ and $t_i=t$ for all $ 1 \le i \le n$, then $HM(G_C)=16n+32mt^2+16mt+8nt^3+16nt^2+16rnt+16nt$
	\end{description}
\end{ex}
\begin{thm} $HM(G_K)=HM(G)+2(r-1)\sum\limits_{v_iv_j \in E(G)}\left[d(v_i)+d(v_j)\right](t_i+t_j)+(r-1)^2\sum\limits_{v_iv_j \in E(G)}(t_i+t_j)^2+(r-1)\sum\limits_{i=1}^{n}t_id(v_i)^{2}+2(r-1)^{2}\sum\limits_{i=1}^{n}d(v_i)t_i(t_i+1)+(r-1)^{3}\sum\limits_{i=1}^{n}\left[t_i^3+2t_i^2+(2r-1)t_i\right]$.
\end{thm}
\begin{proof} \begin{alignat*}{10} HM(G)={}& \sum\limits_{v_iv_j \in E(G_K)}\left[d(v_i)+d(v_j)\right]^2 \\
={}&\sum\limits_{v_iv_j\in E(G)}\left[d(v_i)+t_i(r-1)+d(v_j)+t_j(r-1)\right]^{2} \\
+{}& (r-1)\sum\limits_{i=1}^{n}t_i\left[d(v_i)+t_i(r-1)+(r-1)\right]^{2}+\frac{(r-1)(r-2)}{2}\sum\limits_{i=1}^{n}t_i(2r-2)^{2} \\
={}&HM(G)+2(r-1)\sum\limits_{v_iv_j \in E(G)}\left[d(v_i)+d(v_j)\right]\left[t_i+t_j\right] \\
+{}&(r-1)^2\sum\limits_{v_iv_j\in E(G)}(t_i+t_j)^{2}+(r-1)\sum\limits_{i=1}^{n} t_i\left[d(v_i)^{2}+2(r-1)d(v_i)(t_i+1)+(r-1)^{2}(t_i+1)^{2}\right]\\
+{}&2(r-1)^{3}(r-2)\sum\limits_{i=1}^{n}t_i \\
={}& HM(G)+2(r-1)\sum\limits_{v_iv_j \in E(G)}\left[d(v_i)+d(v_j)\right]\left[t_i+t_j\right] \\
+{}&(r-1)^2\sum\limits_{v_iv_j\in E(G)}(t_i+t_j)^{2}+(r-1)\sum\limits_{i=1}^{n}t_id(v_i)^2 \\
+{}& 2(r-1)^2\sum\limits_{i=1}^{n}d(v_i)t_i(t_i+1)+(r-1)^3\sum\limits_{i=1}^{n}\left[t_i(t_i+1)^{2}+2(r-2)t_i\right] \\
={}& HM(G)+2(r-1)\sum\limits_{v_iv_j \in E(G)}\left[d(v_i)+d(v_j)\right]\left[t_i+t_j\right] \\
+{}&(r-1)^2\sum\limits_{v_iv_j\in E(G)}(t_i+t_j)^{2}+(r-1)\sum\limits_{i=1}^{n}t_id(v_i)^2 \\
+{}& 2(r-1)^2\sum\limits_{i=1}^{n}d(v_i)t_i(t_i+1)+(r-1)^3\sum\limits_{i=1}^{n}\left[t_i^3+2t_i^2+2(r-1)t_i\right]
\end{alignat*}
\end{proof}
\begin{cor} If $t_i=t$, for all $1 \leq i \leq n$, $HM(G_K)=HM(G)+5t(r-1)M_1(G)+4mt(t+1)(r-1)^{2}+n(r-1)^{3}(t^3+2t^2+2(r-1)t)$
\end{cor}
\begin{ex} 
\begin{description}
	\item[(i)] If $G \cong P_n$ and $t_i=t$ for all $1 \le i \le n$, then $HM(G_K)=16n-30+20nrt-20nt-30rt+30t+4mt(t+1)(r-1)^2+n(r-1)^3(t^3+2t^2+2rt-2t)$
	\item [(ii)] If $G \cong C_n$ and $t_i=t$ for all $1 \le i \le n$, then $HM(G)=16n+20nrt-20nt+4mt(t+1)(r-1)^2+n(r-1)^3(t^3+2t^2+2(r-1)t)$ 
\end{description}
	\end{ex}
\begin{thm} $HM(G_A)=HM(G)+2s\sum\limits_{v_iv_j \in E(G)}\left[d(v_i)+d(v_j)\right](t_i+t_j)+ \\ \sum\limits_{i=1}^{n}st_i\left[d(v_i)+r+s^2+r^2+st_i+2rs\right]$ \end{thm}
\begin{proof}
 \begin{alignat*}{3}
HM(G_A)={}& \sum\limits_{v_iv_j \in E(G_A)}\left[d(v_i)+d(v_j)\right]^{2} \\
={}& \sum\limits_{v_iv_j \in E(G)} \left[d(v_i)+t_is+d(v_j)+t_js\right]^{2}+\sum\limits_{i=1}^{n}st_i\left[d(v_i)+t_is+r\right] + \sum\limits_{i=1}^{n}t_is(s+r)^{2} \\
={}& HM(G)+2s\sum\limits_{v_iv_j \in E(G)}\left[d(v_i)+d(v_j)\right](t_i+t_j)+\sum\limits_{i=1}^{n}st_i\left[d(v_i)+r+s^{2}+r^{2}+st_i+2rs\right]
\end{alignat*}
\end{proof}	
\begin{cor} If $t_i=t$, for all $1 \le i \le n$, then $HM(G_A)=HM(G)+4tsM_1(G)+2mst+st(r+s^2+r^2)n+s^2t^2n+2rs^2tn$.  Further, if $r=s$, then $HM(G_A)=HM(G)+4tsM_1(G)+2mst+ns^2t+ns^2t^2+2tns^3$.
\end{cor}
\begin{ex}
\begin{description}
	\item[(i)] If $G \cong P_n,\;t_i=t$ for all $1\le i\le n$  and $r=s$, then $HM(G_A)=16n-30+16nst-24st+2mst+ns^2t+ns^2t^2+2tns^{3}$  
	\item[(ii)] If $G \cong C_n,\;t_i=t$ for all $1 \le i \le n$ and $r=s$, then $HM(G_A)=16n+16nst+2mst+ns^2t+ns^2t^2+2tns^3$
\end{description}

	\end{ex}

\begin{thm} $HM(G_C^{'})=HM(G)+2 \sum\limits_{v_iv_j \in E(G)}\left[d(v_i)+d(v_j)\right](t_i+t_j)+\sum\limits_{v_iv_j \in E(G)}(t_i+t_j)^{2}+M_1(G)+2\sum\limits_{i=1}^{n}d(v_i)t_i+\sum\limits_{i=1}^{n}\left[t_i^{2}+(16r-24)t_i\right]+9n+12m$. \end{thm}

\begin{proof} \begin{alignat*}{7}
HM(G_C^{'})={}& \sum\limits_{v_iv_j \in E(G_C^{'})}\left[d_{G_C^{'}}(v_i)+d_{G_C^{'}}(v_j)\right]^{2}\\
={}& \sum\limits_{v_iv_j \in E(G)}\left[d(v_i)+t_i+d(V_j)+t_j\right]^{2}+\sum\limits_{i=1}^{n}\left[d(v_i)+t_i+3\right]^{2}+50\sum\limits_{i=1}^{n}t_i+16(r-2)\sum\limits_{i=1}^{n}t_i \\
={}& HM(G)+2\sum\limits_{v_iv_j \in E(G)}\left[d(v_i)+d(v_j)\right](t_i+t_j)+\sum\limits_{v_iv_j \in E(G)}(t_i+t_j)^{2}+\sum\limits_{i=1}^{n}d(v_i)^{2} \\
+{}& 2\sum\limits_{i=1}^{n}d(v_i)(t_i+3)+\sum\limits_{i=1}^{n}(t_i+3)^{2}+50\sum\limits_{i=1}^{n}t_i+16(r-2)\sum\limits_{i=1}^{n}t_i \\
={}& HM(G)+2\sum\limits_{v_iv_j \in E(G)}\left[d(v_i)+d(v_j)\right](t_i+t_j)+\sum\limits_{v_iv_j \in E(G)}(t_i+t_j)^{2}+M_1(G)\\+{}& 2\sum\limits_{i=1}^{n}d(v_i)t_i+\sum\limits_{i=1}^{n}\left[t_i^{2}+(16r-24)t_i\right]+9n+12m
\end{alignat*}
\end{proof}
\begin{cor} If $t_i=t$, for all $1 \le i \le n$, then\\ $HM(G_{C}^{'})=HM(G)+(4t+1)M_1(G)+4mt(t+1)+nt^2+nt(16r-24)+12m+9n.$
	\end{cor}
\begin{ex}
	\begin{description}
		\item[(i)] If $G\cong P_n$ and $t_i=t$ for all $1 \le i \le n$, then $HM(G_{C}^{'})=38n-8nt-8t+4mt^2+4mt+16rnt+12m+nt^2-32$
		\item[(ii)] If $G \cong C_n$ and $t_i=t$ for all $1 \le i \le n$, then $HM(G_{C}^{'})=29n-8nt+4mt+16rnt+16m+nt^2$
	\end{description}
\end{ex}
\begin{thm} $HM(G_{K}^{'})=HM(G)+2\sum\limits_{v_iv_j \in E(G)}\left[d(v_i)+d(v_j)\right](t_i+t_j)
	+ \sum\limits_{v_iv_j \in E(G)}(t_i+t_j)^{2}+\sum\limits_{i=1}^{n}t_id(v_i)+\sum\limits_{i=1}^{n}t_i^{2}
	+(2r^{4}-6r^{3}-10r^{2}+15r)\sum\limits_{i=1}^{n}t_i$.
	\end{thm}
\begin{proof} \begin{alignat*}{9}
HM(G_{K}^{'})={}& \sum\limits_{v_iv_j\in E(G_{K}^{'})}\left[d_{G_{K}^{'}}(v_i)+d_{G_{K}^{'}}(v_j)\right]^{2}\\
={}& \sum\limits_{v_iv_j \in E(G)}\left[d(v_i)+t_i+d(v_j)+t_j\right]^{2}+\sum\limits_{i=1}^{n}t_i\left[d(v_i)+t_i+r\right] \\
+{}& 2\sum\limits_{i=1}^{n}t_i(2r-1)^{2}+\left[\frac{r(r-1)}{2}-2\right]\sum\limits_{i=1}^{n}t_i(2r-2)^{2}\\
={}& HM(G)+2\sum\limits_{v_iv_j \in E(G)}\left[d(v_i)+d(v_j)\right](t_i+t_j)+\sum\limits_{v_iv_j \in E(G)}(t_i+t_j)^{2}\\
+{}& \sum\limits_{i=1}^{n}t_id(v_i)+\sum\limits_{i=1}^{n}t_i^{2}
+ r \sum\limits_{i=1}^{n}t_i+2r(r-1)^{3}\sum\limits_{i=1}^{n}t_i-8(r-1)^{2}\sum\limits_{i=1}^{n}t_i\\
={}&  HM(G)+2\sum\limits_{v_iv_j \in E(G)}\left[d(v_i)+d(v_j)\right](t_i+t_j)+\sum\limits_{v_iv_j \in E(G)}(t_i+t_j)^{2}\\
+{}& \sum\limits_{i=1}^{n}t_id(v_i)+\sum\limits_{i=1}^{n}t_i^{2}
+ \left[2r(r-1)^{3}-8(r-1)^{2}+r\right]\sum\limits_{i=1}^{n}t_i\\
Therefore\\
HM(G_{K}^{'})={}&  HM(G)+2\sum\limits_{v_iv_j \in E(G)}\left[d(v_i)+d(v_j)\right](t_i+t_j)+\sum\limits_{v_iv_j \in E(G)}(t_i+t_j)^{2}\\
+{}& \sum\limits_{i=1}^{n}t_id(v_i)+\sum\limits_{i=1}^{n}t_i^{2}
+(2r^4-6r^3-10r^2+15r)\sum\limits_{i=1}^{n}t_i
\end{alignat*}
\end{proof}
\begin{cor} If $t_i=t$, for all $1\le i\le n$, then $HM(G_{K}^{'})=HM(G)+4tM_1(G)+4mt^2+2mt+nt^2+(2r^4-6r^3-10r^2+15r)nt$.	\end{cor}
\begin{ex}
	\begin{description}
		\item[(i)] If $G \cong P_n$ and $t_i=t$ for all $1\le i \le n$, then $HM(G_{K}^{'})=16n-30+16nt-8t+4mt^2+2mt+nt^2+(2r^4-6r^3-10r^2+15r)nt$
		\item[(ii)] If $G \cong C_n$ and $t_i=t$ for all $1 \le i \le n$, then $HM(G_{K}^{'})=16n+16nt+4mt^2+2mt+nt^2+(2r^4-6r^3-10r^2+15r)nt$
	\end{description}
\end{ex}
\begin{thm} $HM(G_{A}^{'})=HM(G)+\sum\limits_{v_iv_j \in E(G)}\left[d(v_i)+d(v_j)\right](t_i+t_j)$\\ $+\sum\limits_{v_iv_j \in E(G)}(t_i+t_j)^{2}+\sum\limits_{i=1}^{n}t_id(v_i)^{2}+2\sum\limits_{i=1}^{n}t_i^{2}d(v_i)+\sum\limits_{i=1}^{n}t_i^{3}+2(s+1)\sum\limits_{i=1}^{n}d(v_i)t_i+2\sum\limits_{i=1}^{n}t_i^{2}+(s^2+3s+2s(s+r)(s+r+1)+1)\sum\limits_{i=1}^{n}t_i$. \end{thm}
\begin{proof} 
\begin{alignat*}{10} 
HM(G_{A}^{'})={}& \sum\limits_{v_iv_j\in E(G_{A}^{'})}\left[d_{G_{A}^{'}}(v_i)+d_{G_{A}^{'}}(v_j)\right]^{2}\\
={}& \sum\limits_{v_iv_j \in E(G)}\left[d(v_i)+t_i+d(v_j)+t_j\right]^{2}+\sum\limits_{i=1}^{n}t_i\left[d(v_i)+t_i+s+1\right]^{2}\\ +{}& \sum\limits_{i=1}^{n}t_is\left[s+r+1\right]^{2}+\sum\limits_{i=1}^{n}t_is(s+r)^{2}\\
={}& HM(G)+2\sum\limits_{v_iv_j \in E(G)}\left[d(v_i)+d(v_j)\right](t_i+t_j)+\sum\limits_{v_iv_j \in E(G)}(t_i+t_j)^{2}\\
+{}&\sum\limits_{i=1}^{n}t_i\left[(d(v_i)+1)^{2}+(t_i+s)^{2}+2(d(v_i)+1)(t_i+s)\right]+\sum\limits_{i=1}^{n}t_is\left[(s+r+1)^{2}+(s+r)^{2}\right] \\
={}& HM(G)+2 \sum\limits_{v_iv_j \in E(G)}\left[d(v_i)+d(v_j)\right](t_i+t_j)+\sum\limits_{v_iv_j \in E(G)}(t_i+t_j)^{2}\\ +{}&\sum\limits_{i=1}^{n}t_i\left[d(v_i)^{2}+2d(v_i)+1\right]\\
+{}&\sum\limits_{i=1}^{n}\left[t_i(t_i^{2}+2st_i+s^{2})\right]+2\sum\limits_{i=1}^{n}\left[t_i(t_i+s)(d(v_i)+1)\right]\\
+{}& \sum\limits_{i=1}^{n}t_is\left\lbrace\left[(s+r)^{2}+2(s+r)+1\right]+(s+r)^{2}\right\rbrace \\
={}& HM(G)+2\sum\limits_{v_iv_j \in E(G)}\left[d(v_i)+d(v_j)\right](t_i+t_j)+\sum\limits_{v_iv_j \in E(G)}(t_i+t_j)^{2}+\sum\limits_{i=1}^{n}t_id(v_i)^{2}\\
+{}& 2\sum\limits_{i=1}^{n}t_i^{2}d(v_i)+\sum\limits_{i=1}^{n}t_i^{3}+2(s+1)\sum\limits_{i=1}^{n}d(v_i)t_i\\
+{}& 2(s+1)\sum\limits_{i=1}^{n}t_i^{2}+(s^2+3s+2s(s+r)(s+r+1)+1)\sum\limits_{i=1}^{n}t_i
\end{alignat*}
\end{proof}
\begin{cor} If $t_i=t$ for all $1\le i \le n$, then $HM(G_{A}^{'})=HM(G)+3tM_1(G)+8mt^2+nt^3+4mt(s+1)+2nt^2(s+1)+(s^2+3s+1)nt+2s(s+r)(s+r+1)nt$.\end{cor}
\begin{ex}
	\begin{description}
		\item[(i)] If $G \cong P_n, t_i=t$ for all $1 \le i \le n$ and $r=s$, then $HM(G_A^{'})=16n-30+13nt-6t+8mt^2+nt^3+4mts+4mt+2nst^2+2nt^2+9nts^2+7nst$.  Further if $r=s=t$, then $HM(G_A^{'})=16n-30+31nr+12nr^3+12mr^2+9nr^2+4mr-6r$
		\item[(ii)] If $G \cong C_n, t_i=t$ for all $1 \le i \le n$ and $r=s$, then $HM(G_A^{'})=16n+13nt+8mt^2+4mt+4mts+9ns^2t+2nst^2+2nt^2+7snt+nt^3$.  Further if $r=s=t$, then $HM(G_A^{'})=16n+13nr+12mr^2+4mr+12nr^3+9nr^2.$
	\end{description}
\end{ex}

\end{document}